\documentclass[11pt]{article}
\usepackage{amsmath}
\usepackage{amssymb,amsfonts,amsmath,amsthm}
\usepackage{epsfig}
\usepackage{color}
\newtheorem{theorem}{Theorem}[section]
\newtheorem{thm}[theorem]{Theorem}
\newtheorem{prop}[theorem]{Proposition}
\newtheorem{lem}[theorem]{Lemma}

\newtheorem{remark}[theorem]{Remark}

\makeatletter \@addtoreset{equation}{section}

\newcommand{\qbinom}[2]{\genfrac{[}{]}{0pt}{}{#1}{#2}}
\def\CC{\mathbb{C}}

\def\CT{\mathop{\mathrm{CT}}}
\def\LC{\mathop{\mathrm{LC}}}

\begin{document}

\title{A Laurent series proof of the Habsieger-Kadell $q$-Morris identity}
\author{Guoce Xin \\ 
\small Department of mathematics\\[-0.8ex]
\small Capital Normal University\\[-0.8ex]
\small Beijing 100048, P.R. China\\
\small\tt guoce.xin@gmail.com\\
\and
Yue Zhou\\ 
\small School of Mathematics and Statistics\\[-0.8ex]
\small Central South University\\[-0.8ex]
\small Changsha 410075, P.R. China\\
\small\tt nkzhouyue@gmail.com
}

\date{March 24, 2014}
\maketitle

\begin{abstract}
We give a Laurent series proof of the Habsieger-Kadell $q$-Morris identity, which is a common generalization of the
$q$-Morris identity and the Aomoto constant term identity. The proof allows us to extend the theorem for some additional
parameter cases.

\end{abstract}

\section{Introduction}
This paper is closely related to the well-known
Dyson's ex-conjecture. The conjecture was made by Freeman Dyson in 1962 when studying statistical theory of energy
levels of complex systems \cite{dyson1962}.
\begin{thm}\label{thm-Dyson}
For nonnegative integers $a_0,\ldots,a_n$,
\begin{align}
\CT_{x}\prod_{0\leqslant i<j\leqslant n} \Big(1-\frac{x_i}{x_j}\Big)^{a_i} \Big(1-\frac{x_j}{x_i}\Big)^{a_j}
= \frac{(a_0+\cdots+a_n)!}{a_0!a_1!\cdots a_n!},
\end{align}
where $\CT_{\mathbf{x}}f(\mathbf{x})$ means to take the constant term in the Laurent expansion of $f(\mathbf{x})$ in the powers of $x_0,x_1,\ldots,x_n$.
\end{thm}
Dyson's ex-conjecture has been proved by many authors using different methods. See, e.g.,
\cite{Gessel-Lv-Xin-Zhou2008,good1,gunson,wilson,zeil1982}. Many variations of Dyson's ex-conjecture have been found,
such as
the famous Macdonald constant term conjectures \cite{Cherednik,macdonald}. Some of them are still not solved. See, e.g., \cite{baker-forr,CLWZ}.

The $q$-analogous of the Dyson conjecture was made by Andrews \cite{andrews1975} in 1975.
\begin{thm}[Zeilberger-Bressoud]\label{thm-dyson}
For nonnegative integers $a_0,a_1,\dots,a_n$,
\begin{align*}
\CT_{\mathbf{x}}\,\prod_{0\leqslant i<j\leqslant n}
\left(\frac{x_i}{x_j}\right)_{\!\!a_i}
\left(\frac{x_j}{x_i}q\right)_{\!\!a_j}
=\frac{(q)_{a_0+a_1+\cdots+a_n}}{(q)_{a_0}(q)_{a_1}\cdots(q)_{a_n}},
\end{align*}
where $(z)_{m}:=\frac{(z;q)_{\infty}}{(zq^m;q)_{\infty}}=(1-z)(1-zq)\cdots (1-z q^{m-1})$.
\end{thm}
Almost all methods for Dyson's ex-conjecture fail to extend for the $q$ version. Up to now, only three different methods succeeded:
the combinatorial proof in \cite{Zeil-Bress1985}, the short proof in \cite{gess-xin2006} using iterated Laurent series, and the one page proof in \cite{karolyi} using the Combinatorial Nullstellensatz. The methods apply to some constant terms of similar type.

In this paper we study the Habsieger-Kadell $q$-Morris identity, an important variation of the equal parameter case of the $q$-Dyson theorem.
The original identity studies the constant term of the following Laurent polynomial for $m+l\leqslant n$:
\begin{align}\label{eq-qM}
A_{q}&(x_{0},x_{1},\ldots,x_{n};a,b,k,m,l)\nonumber \\
&=\prod_{i=1}^{n}\Big(\frac{q^{\chi(i\leqslant
m)}x_{0}}{x_{i}}\Big)_{a-\chi(i\leqslant m)}
\Big(\frac{q^{\chi(i>m)}x_{i}}{x_{0}}\Big)_{b+\chi(i\leqslant m)+\chi(i\geqslant n-l+1)}\prod_{1\leqslant i<j\leqslant n}
\Big(\frac{x_{i}}{x_{j}}\Big)_{k}\Big(\frac{x_{j}}{x_{i}}q\Big)_{k},
\end{align}
where the expression $\chi(S)$ is $1$ if the statement $S$ is true,
and $0$ otherwise.

In giving a Laurent series proof of the Habsieger-Kadell $q$-Morris identity,
we are able to establish a unified formula that also works for the additional cases $m+l>n$. The result is stated as follows, where
the additional boldfaced part $\mathbf{\chi(i\geqslant 2n-m-l)}$ is only effective when $m+l>n$.
\begin{thm}
\label{t-HabsigerKadell}
For nonnegative integers $a,b,k,m,l$ satisfying $m,l< n$, we have
\begin{align}\label{eq-qMorris}
\CT_{x}A_{q}(x_{0},x_{1},\ldots,x_{n};a,b,k,m,l)=M_{n}(a,b,k,m,l;q),
\end{align}
where
\begin{align}
M_{n}(a,b,k,m,l;q)=\prod_{i=0}^{n-1}\frac{(q)_{a+b+ik+\chi(i\geqslant n-l)}(q)_{(i+1)k}}{(q)_{a+ik-\chi(i<m)}(q)_{b+ik+\chi(i\geqslant n-m-l)+\mathbf{\chi(i\geqslant 2n-m-l)}}(q)_{k}}.
\end{align}
\end{thm}


The $m=l=0,\ q=1$ case of the result is the Morris identity, which is equivalent to the well-known Selberg integral \cite{Selberg}. In his thesis \cite{Morris1982} Morris established the identity and conjectured the $q$-analogous identity. The $q$-Morris identity, or the $m=l=0$ case, was proved by Habsieger \cite{Habsieger} and later by Zeilberger \cite{Zeilberger1989}.
The $m=0,\ q=1$ case of the result, called the Aomoto identity, was constructed by Aomoto \cite{Aomoto}. By extending Aomoto's method Kadell \cite{Kadell}
constructed the $m+l\leqslant n$ case, in the same year of Habsieger's proof. As far as we know, the $m+l>n$ case was not considered before.

Our approach is by extending the proof of the Aomoto identity in \cite{Gessel-Lv-Xin-Zhou2008}.
The basic idea is to regard both sides of
\eqref{eq-qMorris} as polynomials in $q^a$ of degree at most $d=nb+m+l$. Then to show the equality of the two polynomials, it is sufficient to show that they are equal at $d+1$ points. The equality at the $d$ vanishing points are not hard to handle by the techniques in \cite{gess-xin2006,Lv-Xin-Zhou2009}. But in this approach, we have to deal with two problems: i) the multiple roots problem for small $k$; ii) the $d+1$-st suitable point is hard to find. We handle the former problem by a rationality result of Stembridge, and the latter problem by a hard searching process.

We present the major steps of our proof in Section \ref{s-2}. The steps are expanded by the rationality result in Section \ref{s-2.3}, by the proof of the vanishing lemma in Section \ref{s-2.4}, and by the proof for the extra point in Section \ref{s-extra}.


While we were finishing the presented work, the one page proof of the $q$-Dyson theorem was published. Moreover, K\'{a}rolyi and Nagy \cite{Nagy} found a short proof of Theorem \ref{t-HabsigerKadell} in the $l=0$ case using the Combinatorial Nullstellensatz. The two approaches are different but have some connections.

\section{Proof of the Habsieger-Kadell $q$-Morris identity}\label{s-2}

Following notations in the introduction, we may assume that $0\leqslant m,l<n$ by the following argument.
If $m\geqslant n$ then
$$A_{q}(x_0,\ldots,x_n;a,b,k,m,l)
=\prod_{i=1}^{n}\Big(\frac{qx_{0}}{x_{i}}\Big)_{a-1}
\Big(\frac{x_{i}}{x_{0}}\Big)_{b+1+\chi(i\geqslant n-l+1)}\prod_{1\leqslant i<j\leqslant n}
\Big(\frac{x_{i}}{x_{j}}\Big)_{k}\Big(\frac{x_{j}}{x_{i}}q\Big)_{k},$$
which is just $A_{q}(x_0q,\ldots,x_n;a-1,b+1,k,0,l)$. Then by substituting $x_0$ by $x_0/q$, we can see that the constant term is $M_n(a-1,b+1,k,0,l;q)$.
The case $l\geqslant n$ is similar: we observe that
$A_{q}(x_0,\ldots,x_n;a,b,k,m,l)$ can be rewritten as  $A_q(x_0,\ldots,x_n;a,b+1,k,m,0)$.

Let us rewrite $M(q^a,q^k)=M_{n}(a,b,k,m,l;q)$ as
\begin{align}
M(q^a,q^k)
=&\frac{(q)_{nk}}{(q)_k^n } \cdot \prod_{i=0}^{m-1}(1-q^{a+ik})\cdot \prod_{i=n-l}^{n-1}(1-q^{a+ik+b+1}) \nonumber \\
&\cdot \prod_{i=0}^{n-1}\frac{(1-q^{a+ik+1})(1-q^{a+ik+2})\cdots (1-q^{a+ik+b})(q)_{ik}}{(q)_{b+ik+\chi(i\geqslant n-m-l)+\chi(i\geqslant 2n-m-l)}}.
\end{align}

We have the following characterization.
\begin{lem}
For fixed $b, n \in \mathbb{N}$ and $0\leqslant m,l<n$,  $M(q^a, q^k)$ is uniquely determined by the following three properties.
\begin{enumerate}
\item  $M(q^a,q^k)(q)_k^n/(q)_{nk}$ is   a polynomial in $q^a$ of degree $nb+m+l$, whose coefficients are rational functions in $q^k$ and $q$.

\item For any $k>b+1$, $M(q^{-\xi}, q^k)=0$ if $\xi$ belongs to one of the following three sets:
\begin{align}\label{root}
D_1=&\{0,k,\ldots,(m-1)k\}; \nonumber \\
D_2=&\{(n-l)k+b+1,(n-l+1)k+b+1\ldots,(n-1)k+b+1\}; \\
D_3=&\{ik+1,ik+2\ldots,ik+b\mid i=0,\ldots,n-1\}.\nonumber
\end{align}

\item For any $k>b+1$ we have $M(q^{-(n-l-1)k-b-1}, q^k) = M_{n}(-(n-l-1)k-b-1,b,k,m,l;q).$
\end{enumerate}
\end{lem}
\begin{proof}
Assume $M'(q^a,q^k)$ also satisfies the above three properties.  Then for every $\xi \in D_i, i=1,2,3$ or $\xi = (n-l-1)k+b+1$,
$$ M(q^{-\xi}, q^k)(q)_k^n/(q)_{nk}= M'(q^{-\xi}, q^k)(q)_k^n/(q)_{nk}, \qquad \text{ for all } k>b+1.$$
Since both sides are rational functions in $q^k$ and they agree at infinitely many points, they are identical as rational functions.

Now as polynomials in $q^a$, whose coefficients are rational functions in $q^k$ and $q$, $M(q^{a}, q^k)(q)_k^n/(q)_{nk}$ agrees with $M'(q^{a}, q^k)(q)_k^n/(q)_{nk}$ at $nb+m+l+1$ distinct $\xi$'s as above,
they must be equal to each other.
\end{proof}

Note that the condition $k>b+1$ can not be dropped, since $D_3$ has duplicate elements when $k\leqslant b-1$, and $D_3$ or $D_2$ intersects $D_1$ if
$k=b$ or $k=b+1$.

Denote by $M'_{n}(a,b,k,m,l;q)$ the left-hand-side of \eqref{eq-qMorris}. Then Theorem \ref{t-HabsigerKadell} will follow by induction on $n$ if we can show
the following three lemmas, whose proofs will be given in later sections.

\begin{lem}\label{lem-polynomial}
For fixed $b, n \in \mathbb{N}$ and $0\leqslant m,l<n$, $M'_{n}(a,b,k,m,l;q)(q)_k^n/(q)_{nk}$
is a polynomial in $q^{a}$ of degree at most $nb+m+l$, whose coefficients are rational functions in $q^k$ and $q$.
\end{lem}

Since $M'_n(a,b,k,m,l;q)$ is a polynomial in $q^a$,
the definition of $a$ can be extended for all integers, in particular for negative integers $a$.

\begin{lem}[Vanishing Lemma]\label{lem-vanish}
For fixed $b, n \in \mathbb{N}$, and $0\leqslant m,l<n$, and $k>b+1$, $ M'_n(-h,b,k,m,l;q)$ vanishes when $h$ equals one of the
values in \eqref{root}.
\end{lem}

\begin{remark}\label{remark1}
At this stage we can already claim the truth of Theorem \ref{t-HabsigerKadell} for $m=0$. In this case $D_1$ is empty, so that we can choose
$a=0$ as the extra point. Then $M'_n(0,b,k,0,l;q)$ reduces to $M'_n(0,0,k,0,0;q)$, and the equal parameter case of the $q$-Dyson theorem
applies.
\end{remark}

\begin{lem}\label{lem-extra}
 For fixed $b, n \in \mathbb{N}$, and $0\leqslant m,l<n$, and $k>b+1$,
 if we assume Theorem \ref{t-HabsigerKadell} holds for smaller values of $n$,
then $M'_n(-h,b,k,m,l;q)=M_n(-h,b,k,m,l;q)$ when
$h=(n-l-1)k+b+1$.
\end{lem}
The extra point in the above lemma is found through a hard searching process. It is a surprise for this special $h$:
the constant term  $ M'_n(-h,b,k,m,l;q)$ reduces to a single constant term that can be evaluated by Remark \ref{remark1} or the hypothsis.

\section{The polynomial-rational characterization}\label{s-2.3}
To prove Lemma \ref{lem-polynomial}, we need the the following rationality result, which is implicitly due to
J.R. Stembridge \cite{stembridge1987}, as can be seen from the proof. The $q=1$ case of this result is the equal parameter case of \cite[Proposition 2.4]{Gessel-Lv-Xin-Zhou2008}.
\begin{prop}\label{thm-rationality}
For any $n\in \mathbb{N}$ and $\alpha=(\alpha_1,\ldots,\alpha_n)\in \mathbb{Z}^{n}$ with
$\sum_{1\leqslant i\leqslant n} \alpha_i=0$,  we have
\begin{align}
[x^{\alpha}]\prod_{1\leqslant i<j\leqslant n}
\Big(\frac{x_{i}}{x_{j}}\Big)_{k}\Big(\frac{x_{j}}{x_{i}}q\Big)_{k}=\frac{(q)_{nk}}{(q)_{k}^{n}}\cdot
R_n(q^k;q).
\end{align}
where $R_n(q^k;q)$ is a rational function in $q^k$ and $q$, and
$[x^{\alpha}]$ refers to take the coefficient of
$x_1^{\alpha_1}\cdots x_n^{\alpha_n}$ in the polynomial.
\end{prop}
\begin{proof}
In \cite[Equation 44]{stembridge1987} Stembridge gave the following
equation (set $z=q^k$)
\begin{align}\label{eq-1}
[x^{\alpha}]\prod_{1\leqslant i<j\leqslant n}
\Big(\frac{x_{i}}{x_{j}}\Big)_{k}\Big(\frac{x_{j}}{x_{i}}q\Big)_{k}=\frac{1}{(q)_{k-1}^{n}}\sum_{S}\pm (-q^k)^{|S|}C^{n}[S](q^k,q),
\end{align}
where the summation is taken over some elements whose number
is bounded by a function of $n$
and
$C^{n}[S](q^k,q)$ is a formal power series in $q^k$ and $q$.

By
\cite[Corollary 3.3]{stembridge1987} we know that
\begin{align}\label{eq-2}
C^{n}[\varnothing](q^k,q)= \frac{(q)_{nk}}{(1-q^k)(1-q^{2k})\cdots (1-q^{nk})}.
\end{align}

In \cite[Page 334, Line 33]{stembridge1987} Stembridge stated that
$C^{n}[\lambda](q^k,q)$ is of the form $f_{\lambda}(q^k,q)\cdot
C^{n}[\varnothing](q^k,q)$ for some rational function $f_\lambda$.
Therefore, combining with \eqref{eq-1} and \eqref{eq-2} we get
\begin{align}\label{eq-H2}
[x^{\alpha}]\prod_{1\leqslant i<j\leqslant n}
\Big(\frac{x_{i}}{x_{j}}\Big)_{k}\Big(\frac{x_{j}}{x_{i}}q\Big)_{k}=\frac{(q)_{nk}}{(q)^n_{k-1}(1-q^k)\cdots (1-q^{nk})}
\sum_{S}\pm (-q^k)^{|S|}f_{S}(q^k,q).
\end{align}
The desired rational function is then given by
\begin{align*}
R_n(q^{k},q)=\frac{(1-q^k)^n}{(1-q^k)(1-q^{2k})\cdots (1-q^{nk})}\sum_{S}\pm (-q^{k})^{|S|}f_{S}(q^{k},q).
\end{align*}
\end{proof}

\begin{proof}[Proof of Lemma \ref{lem-polynomial}]
When regarded as Laurent
series in $x_0$, the equality
\begin{align*}
\left(\frac{x_0}{x_i} \right)_{\!\!a}\!\left(\frac{x_i}{x_0}q
\right)_{\!\!b}&=q^{\binom{b+1}{2}}\left(-\frac{x_i}{x_0}\right)^{\!\!b}
\!\left(\frac{x_0}{x_i} q^{-b}\right)_{\!\!a+b}
\end{align*}
can be easily shown to hold for all integers $a$. Rewrite
$M'_n(a,b,k,m,l;q)$ as
\begin{align}\label{e-product}
\CT_x \prod_{i=1}^n q^{b^*_i+1 \choose 2}\left(-\frac{x_i}{x_0}\right)^{b^{*}_i+\chi(i\leqslant m)}
\left(\frac{x_0}{x_i}q^{-b^{*}_i}\right)_{a+b^{*}_i}
\prod_{1\leqslant i<j\leqslant n}
\Big(\frac{x_{i}}{x_{j}}\Big)_{k}\Big(\frac{x_{j}}{x_{i}}q\Big)_{k},
\end{align}
where $b^*_i=b+\chi(i\geqslant n-l+1)$.

The well-known $q$-binomial theorem \cite[Theorem
2.1]{andrew-qbinomial} is the identity
\begin{align}
\label{e-qbinomial} \frac{(bz)_\infty}{(z)_\infty} =
\sum_{k=0}^\infty \frac{(b)_k}{(q)_k} z^k.
\end{align}
Setting $z=uq^n$ and $b=q^{-n}$ in \eqref{e-qbinomial}, we obtain
\begin{align}\label{e-qbinomialn}
(u)_n=\frac{(u)_\infty}{(uq^n)_\infty}=
\sum_{k=0}^\infty q^{k(k-1)/2}\qbinom{n}{k} (-u)^k
\end{align}
for all integers $n$, where $\qbinom{n}{k}=\frac{(q)_n}{(q)_k(q)_{n-k}}$ is the $q$-binomial coefficient.

Using \eqref{e-qbinomialn}, we see
that for $1\leqslant i\leqslant n$, {\small
\begin{align*} q^{b^*_i+1 \choose 2}\left(-\frac{x_i}{x_0}\right)^{b^{*}_i+\chi(i\leqslant m)}
\left(\frac{x_0}{x_i}q^{-b^{*}_i}\right)_{a+b^{*}_i}
=\sum_{k_i\geqslant 0}C(k_i) \qbinom{a+b^*_i}{k_i}
x_0^{k_i-b^*_i-\chi(i\leqslant m)}x_i^{b^*_i+\chi(i\leqslant m)-k_i},
\end{align*}}
where $C(k_i)=(-1)^{k_i+b^*_i+\chi(i\leqslant m)}q^{\binom{b^*_i+1}2 + \binom {k_i}2
-k_ib^*_i}$.

Expanding the first product in \eqref{e-product} and  taking constant
term in $x_0$, we see that, by Proposition \ref{thm-rationality}, $M'_n(a,b,k,m,l;q)$ becomes
\begin{align} \label{e-midle} \sum_{\mathbf{k}}
\prod_{i=1}^n \qbinom{a+b^*_i}{k_i} \CT_{x_1,\dots,x_n} L(x_1,\dots
,x_n;\mathbf{k})= \frac{(q)_{nk}}{(q)_k^n} \sum_{\mathbf{k}}
\prod_{i=1}^n \qbinom{a+b^*_i}{k_i}  R(q^k,q; \mathbf{k})  ,
\end{align}
for some rational functions $R(q^k,q; \mathbf{k}) $ in $q^k$ and $q$, where
\begin{align}\label{e-L}
L(x_1, \dots, x_n;\mathbf{k})=
&q^{(n-l)\binom{b+1}{2}+l\binom{b+2}{2}+\sum_{i=1}^n\binom{k_i}{2}
-b\sum_{i=1}^{n-l}k_i-(b+1)\sum_{i=n-l+1}^n k_i} \nonumber \\
&\cdot \prod_{i=1}^m x_i^{b^*_i+1-k_i}
\prod_{i=m+1}^n x_{i}^{b^*_i-k_{i}}
\prod_{1\leqslant i<j\leqslant n}
\Big(\frac{x_{i}}{x_{j}}\Big)_{k}\Big(\frac{x_{j}}{x_{i}}q\Big)_{k}
\end{align}
is a Laurent
polynomial in $x_1, \dots, x_n$ independent of $a$ and the sum
ranges over all sequences $\mathbf{k}=(k_1,\dots, k_n)$ of
nonnegative integers satisfying $k_1+k_2+\cdots+k_n=nb+m+l.$ Since
$\qbinom{a+b^*_i}{k_i}$ is a polynomial in $q^{a}$ of degree
$k_i$, each summand in \eqref{e-midle} is a polynomial in $q^{a}$
of degree at most $k_1+k_2+\cdots +k_n=nb+m+l$, and so is the sum.

The coefficients of $M'_n(a,b,k,m,l;q)(q)_k^n/ (q)_{nk}   $ in $q^a$ are clearly rational functions in $q^k$ and $q$. 
\end{proof}

\section{Proof of the vanishing lemma}\label{s-2.4}
We will follow notations in \cite{gess-xin2006,Lv-Xin-Zhou2009}, where different versions of the vanishing lemma were proposed for dealing with $q$-Dyson related constant terms.
The new vanishing lemma will be handled by the same idea but we have to carry out the details. We
will include some basic ingredients for readers' convenience.

In this section, we let $K=\CC(q)$, and assume that all series are in
the field of iterated Laurent series $K(\!(x_n)\!)(\!(x_{n-1})\!)\cdots (\!(x_0)\!)$. The reason for choosing
$K(\!(x_n)\!)(\!(x_{n-1})\!)\cdots (\!(x_0)\!)$ as a
working field has been explained in \cite{gess-xin2006}.

We emphasize that the field of rational functions is a subfield of
$K(\!(x_n)\!)(\!(x_{n-1})\!)\cdots (\!(x_0)\!)$, so that
every rational function is identified with its unique iterated
Laurent series expansion. The series expansions of $1/(1-q^k
x_i/x_j)$ will be especially important.
\begin{align*}
  \frac{1}{1-q^k x_i/x_j}&=\sum_{l= 0}^\infty q^{kl} x_i^l x_j^{-l}, \text{ if } i<j,\\
\frac{1}{1-q^k x_i/x_j} & =\frac1{-q^k
x_i/x_j(1-q^{-k}x_j/x_i)}
=\sum_{l=0}^\infty -q^{-k(l+1)} x_i^{-l-1}x_j^{l+1}, \text{ if } i>j.
\end{align*}
The constant term of the series $F(\mathbf{x})$ in $x_i$, denoted by
$\CT_{x_i} F(\mathbf{x})$, is defined to be the sum of those terms
in $F(\mathbf{x})$ that are free of $x_i$. It follows that
\begin{equation}
\label{e-ct} \CT_{x_i} \frac{1}{1-q^k x_i/x_j} =
\begin{cases}
    1, & \text{ if }i<j, \\
    0, & \text{ if }i>j. \\
\end{cases}
\end{equation}
We shall call the monomial $M=q^k x_i/x_j$ \emph{small} if $i<j$ and
\emph{large} if $i>j$.  Thus the constant term in $x_i$ of $1/(1-M)$
is $1$ if $M$ is small and $0$ if $M$ is large.

Constant term operators defined in this
way has the important commutativity property:
$$\CT_{x_i} \CT _{x_j} F(\mathbf{x}) = \CT_{x_j} \CT_{x_i} F(\mathbf{x}).$$

The \emph{degree} of a rational function of $x$ is the degree in $x$
of the numerator minus the degree  in $x$ of the denominator. For
example, if $i\ne j$ then  the degree  of $1-x_j/x_i=(x_i-x_j)/x_i$
is $0$ in $x_i$ and $1$ in $x_j$. A rational function is called
\emph{proper} (resp. \emph{almost proper}) in $x$ if its degree in
$x$ is negative (resp. zero).

Let
\begin{align}\label{e-defF}
F=\frac{p(x_k)}{x_k ^d \prod_{i=1}^m (1-x_k/\alpha_i)}
\end{align}
be a rational function of $x_k$, where $p(x_k)$ is a polynomial in
$x_k$, and the $\alpha_i$ are distinct monomials, each of the form
$x_t q^s$. Then the partial fraction decomposition of $F$ with
respect to $x_k$ has the following form:
{\small\begin{align}\label{e-defFs}
F=p_0(x_k)+\frac{p_1(x_k)}{x_k^d}+\sum_{j=1}^m \frac{1}{1-
x_k/\alpha_j}  \left. \left(\frac{p(x_k)}{x_k^d \prod_{i=1,i\ne j}^m
(1-x_k/\alpha_i)}\right)\right|_{x_k=\alpha_j},
\end{align}}
where $p_0(x_k)$
is a polynomial in $x_k$, and $p_1(x_k)$ is a polynomial in $x_k$ of
degree less than $d$.

The following lemma has appeared in \cite{Lv-Xin-Zhou2009}.
\begin{lem}\label{lem-almostprop}
Let $F$ be as in \eqref{e-defF} and \eqref{e-defFs}.
 Then
\begin{align}\label{e-almostprop}
\CT_{x_k} F=p_0(0) +\sum_j  \bigl(F\,
(1-x_k/\alpha_j)\bigr)\Bigr|_{x_k =\alpha_j},
\end{align}
where 
the sum ranges over all $j$ such that $x_k/\alpha_j$ is small. In
particular, if $F$ is proper in $x_k$, then $p_0(x_k)=0$; if $F$ is
almost proper in $x_{k}$, then
$p_0(x_k)=(-1)^m\prod_{i=1}^m\alpha_{i}\LC_{x_{k}}p(x_k)$, where
$\LC_{x_k}$ means to take the leading coefficient with respect to
$x_k$.
\end{lem}

The following lemma plays an important role in our argument.
\begin{lem}\label{lem-important}
Let $k$, $b$ and $k_1,\ldots,k_s$ be nonnegative integers. Then for any
$k_{1},\ldots,k_{s}$ with $0\leqslant k_{i} \leqslant (s-1)k+b+1$ for all
$i$, either $0\leqslant k_{i}\leqslant b$ for some $i$, or $1-k\leqslant k_{j}-k_{i}\leqslant k$ for
some $i<j$, except only when $k_{i}=(s-i)k+b+1$ for $i=1,\ldots,s$.
\end{lem}
\begin{proof}
Assume $k_1,\ldots,k_s$ to satisfy that for all $i$,
$b<k_{i}\leqslant (s-1)k+b+1$, and for all $i<j$,
either $k_j-k_i>k$ or $k_j-k_i\leqslant -k$.
Then we need to show that $k_{i}=(s-i)k+b+1$ for $i=1,\ldots,s$.

 Let $[b,c]$ denote the set $\{b,b+1,\ldots,c\}$ for integers $b\le c$.
 Divide $[b+1,(s-1)k+b+1]$ into the following $s-1$ parts:
 $[b+1,k+b+1]$, $[k+b+2,2k+b+1],\ldots, [(s-2)k+b+2,(s-1)k+b+1]$.
By the assumption that $b<k_{i}\leqslant (s-1)k+b+1$
for all the $s$ $k_i$'s, it follows that at least
two of the $k_i$'s have to be in a common range, let the two be $k_{t_1}$ and $k_{t_2}$ and assume $t_2<t_1$.
It follows that $k_{t_1}$ and $k_{t_2}$ must belong to $[b+1,k+b+1]$ otherwise
if $k_{t_1},k_{t_2}\in [(i-1)k+b+2,ik+b+1]$ for some $2\leqslant i\leqslant s-1$, then we have
$1-k\leqslant k_{t_1}-k_{t_2}\leqslant k$ and it contradicts
to our assumption. Furthermore,
we have $k_{t_1}=b+1,k_{t_2}=k+b+1$ and
the remaind $s-2$ $k_i$'s must be distributed to the $s-2$ ranges($[(i-1)k+b+2,ik+b+1]$ for $i=2,\ldots,s-1$) averagely,
otherwise it contradicts to the assumption.
Let $k_{t_3}\in [k+b+2,2k+b+1]$.
Then $k_{t_3}=2k+b+1$ and $t_3<t_2$ otherwise $1-k\leqslant k_{t_2}-k_{t_3}\leqslant k$.
Let $k_{t_4}\in [2k+b+1,3k+b+1]$. Then $k_{t_4}=3k+b+1$ and $t_3<t_4$ otherwise $1-k\leqslant k_{t_3}-k_{t_4}\leqslant k$.
Following this discussion, we have $k_{t_i}=(i-1)k+b+1$ for $i=1,\ldots,s$ and $t_s<\cdots<t_1$. Thus we have $t_i=(s-i+1)$ and
$k_{i}=(s-i)k+b+1$ for $i=1,\ldots,s$.
\end{proof}

Let
\begin{align}\label{eq-qM}
Q(h)=&\prod_{i=1}^{m}\frac{\Big(x_{i}/x_{0}\Big)_{b+1+\chi(i\geqslant n-l+1)}}{(1-x_{0}/x_{i})\cdots(1-x_{0}/(x_{i}q^{h}))}
\prod_{i=m+1}^{n}\frac{\Big(x_{i}q/x_{0}\Big)_{b+\chi(i\geqslant n-l+1)}}{(1-x_{0}/(x_{i}q))\cdots(1-x_{0}/(x_{i}q^{h}))}\nonumber \\
\cdot &\prod_{1\leqslant i<j\leqslant n} \Big(\frac{x_{i}}{x_{j}}\Big)_{k}\Big(\frac{x_{j}}{x_{i}}q\Big)_{k}.
\end{align}
By the proof of Lemma \ref{lem-polynomial} in Section \ref{s-2.3}, we have
$$ M'_n(-h,b,k,m,l;q) = \CT_{\mathbf{x}}Q(h).$$
The vanishing lemma says that $\CT_{\mathbf{x}}Q(h)=0$ for every $h$ in \eqref{root}.

We attack the vanishing lemma by repeated application of Lemma \ref{lem-almostprop}. This will give a big sum of terms, each will be
detected to be $0$ by Lemma \ref{lem-important}. This is better summarized in the following Lemma \ref{lem-lem}. To state the lemma, we need more notations.

For any rational function $F$ of $x_0, x_1, \ldots , x_n$, and for
sequences of integers $k = (k_1, k_2, \ldots, k_s)$ and $r = (r_1,
r_2, \ldots, r_s)$ let $E_{\mathbf{r},\mathbf{k}}F$ be the result of
replacing $x_{r_i}$ in $F$ with $x_{r_s}q^{k_s-k_i}$ for $i = 0,
1,\ldots, s-1$, where we set $r_0 = k_0 = 0$. Then for $0 < r_1 <
r_2 <\ldots
 < r_s \leqslant n$ and $0\leqslant k_i\leqslant h$, we define
\begin{equation}\label{eq-Qrk}
Q(h\mid \mathbf{r};\mathbf{k})=Q(h\mid
r_{1},\ldots,r_{s};k_{1},\ldots,k_{s})=E_{\mathbf{r},\mathbf{k}}
\left[Q(h)\prod_{i=1}^{s}(1-\frac{x_{0}}{x_{r_{i}}q^{k_{i}}})\right].
\end{equation}
Note that the product on the right hand side of \eqref{eq-Qrk}
cancels all the factors in the denominator of $Q$ that would be
taken to zero by $E_{\mathbf{r},\mathbf{k}}$.
If $k_i=0$ for some $i$ and $r_i\leqslant m$,
then $Q(h\mid \mathbf{r};\mathbf{k})$ has the factor $E_{\mathbf{r},\mathbf{k}}[(x_{r_i}/x_0)_{b+1+\chi(r_i\geqslant n-l+1)}]=0$.
If $k_i=0$ for some $i$ and $r_i>m$,
by the definition of $Q(h\mid \mathbf{r};\mathbf{k})$ in \eqref{eq-Qrk},
the factor $1-x_0/x_{r_i}$ appears in $Q(h\mid \mathbf{r};\mathbf{k})$,
but it cancels nothing in the denominator of $Q(h)$. Thus it would be taken to
zero by $E_{\mathbf{r},\mathbf{k}}$ and  $Q(h\mid \mathbf{r};\mathbf{k})=0$.
Therefore, if $k_i=0$ for some $i$, then $Q(h\mid \mathbf{r};\mathbf{k})=0$.

As a warm up, it
is easy to check that $Q(h)$ is proper in $x_{0}$ with degree $-nh-m$.
Thus applying Lemma \ref{lem-almostprop} gives
\begin{equation}\label{eq-Qx}
\CT_{x_{0}}Q(h)=\sum_{\begin{subarray}{l} 1\leqslant r_{1,1}\leqslant m \\ 0\leqslant k_{1,1}\leqslant h\end{subarray}}Q(h\mid r_{1,1};k_{1,1})
+\sum_{\begin{subarray}{c} m+1\leqslant r_{2,1}\leqslant n \\ 1\leqslant k_{2,1}\leqslant h\end{subarray}}Q(h\mid r_{2,1};k_{2,1}).
\end{equation}
Since $Q(h\mid r_{1,1};0)=0$, we can rewrite
\eqref{eq-Qx} as
\begin{equation}\label{eq-Qx0}
\CT_{x_{0}}Q(h)=\sum_{\begin{subarray}{l} 1\leqslant r_{1}\leqslant n \\ 1\leqslant
k_{1}\leqslant h\end{subarray}}Q(h\mid r_{1};k_{1}).
\end{equation}
This formula is compatible with the following lemma if we treat $Q(h)=Q(h\mid \varnothing;\varnothing)$.

\begin{lem}\label{lem-lem}
The rational functions $Q(h\mid \mathbf{r};\mathbf{k})$ have the
following two properties:
\begin{itemize}
\item[\emph{(i)}] If $0\leqslant k_{i}\leqslant (s-1)k+b+\chi(s\geqslant n-l+1)$ for all $i$ with $1\leqslant i\leqslant s$, then $Q(h\mid \mathbf{r};\mathbf{k})=0$.

\item[\emph{(ii)}] Suppose $k>b+1$ and $h\in D_1\bigcup D_2 \bigcup D_3\bigcup \{(n-l-1)k+b+1\}$. If $k_{i}> (s-1)k+b+\chi(s\geqslant n-l+1)$ for some $i$ with $1\leqslant i\leqslant s$ and $n>s$, then
\begin{equation}\label{eq-s-s+1}
\CT_{x_{r_s}}Q(h\mid \mathbf{r};\mathbf{k})=\sum_{\begin{subarray}{l}
r_{s}<r_{s+1}\leqslant n \\ 0\leqslant k_{s+1}\leqslant h \end{subarray}}Q(h\mid
r_{1},\ldots,r_{s},r_{s+1};k_{1},\ldots,k_{s},k_{s+1}).
\end{equation}
\end{itemize}
\end{lem}

\emph{Proof of property (i)}. By Lemma \ref{lem-important},
if $0\leqslant k_{i}\leqslant (s-1)k+b+\chi(s\geqslant n-l+1)$ for all $i$,
the $k_i$'s have to be in one of the following three cases.
Case 1:
$0\leqslant k_{i}\leqslant b$ for some $1\leqslant i\leqslant s$;
Case 2: $1-k\leqslant k_{j}-k_{i}\leqslant k$ for some $i<j$;
Case 3: $k_{i}=(s-i)k+b+1$ for $i=1,\ldots,s$.
Note that Case 3 occurs only when $s\geqslant n-l+1$.

Case 1:
$0\leqslant k_{i}\leqslant b$ for some $1\leqslant i\leqslant s$.
If $r_{i}\leqslant m$, then $Q(h\mid \mathbf{r};\mathbf{k})$
has the factor
\begin{equation*}
E_{\mathbf{r},\mathbf{k}}\left[\Big(\frac{x_{r_{i}}}{x_{0}}\Big)_{b+1+\chi(r_i\geqslant n-l+1)}\right]=\left(\frac{x_{r_{s}}q^{k_{s}-k_{i}}}{x_{r_{s}}q^{k_{s}}}\right)_{b+1+\chi(r_i\geqslant n-l+1)}=(q^{-k_{i}})_{b+1+\chi(r_i\geqslant n-l+1)}=0.
\end{equation*}
If $r_{i}>m$, then $Q(h\mid \mathbf{r};\mathbf{k})=0$ for $k_i=0$ and for $1\leqslant k_{i}\leqslant b$ it
has the factor
\begin{equation*}
E_{\mathbf{r},\mathbf{k}}\left[\Big(\frac{x_{r_{i}}q}{x_{0}}\Big)_{b+\chi(r_i\geqslant n-l+1)}\right]=\left(\frac{x_{r_{s}}q^{1+k_{s}-k_{i}}}{x_{r_{s}}q^{k_{s}}}\right)_{b+\chi(r_i\geqslant n-l+1)}
=(q^{1-k_{i}})_{b+\chi(r_i\geqslant n-l+1)}=0.
\end{equation*}

Case 2: $1-k\leqslant k_{j}-k_{i}\leqslant k$ for some $i<j$. In this case $Q(h\mid
\mathbf{r};\mathbf{k})$ has the factor
\begin{equation*}
E_{\mathbf{r},\mathbf{k}}\left[\Big(\frac{x_{r_{i}}}{x_{r_{j}}}\Big)_{k}
\Big(\frac{x_{r_{j}}}{x_{r_{i}}}q\Big)_{k}\right],
\end{equation*}
which is equal to
\begin{equation*}
E_{\mathbf{r},\mathbf{k}}\left[q^{{k+1\choose
2}}\left(-\frac{x_{r_{j}}}{x_{r_{i}}}\right)^k\left(\frac{x_{r_{i}}}{x_{r_{j}}}q^{-k}\right)_{2k}\right]
=q^{{k+1\choose 2}}(-q^{k_{i}-k_{j}})^{k}(q^{k_{j}-k_{i}-k})_{2k}=0.
\end{equation*}

Case 3: $k_{i}=(s-i)k+b+1$ for $i=1,\ldots,s$. In this case we only need the value of $k_{s}$. Since Case 3 only occurs when $s\geqslant n-l+1$,
we have $r_s\geqslant s\geqslant n-l+1$.
If $r_s\leqslant m$, then $Q(h\mid \mathbf{r};\mathbf{k})$ has the factor
\begin{equation*}
E_{\mathbf{r},\mathbf{k}}\left[\Big(\frac{x_{r_{s}}}{x_{0}}\Big)_{b+1+\chi(r_s\geqslant n-l+1)}\right]=
E_{\mathbf{r},\mathbf{k}}\left[\Big(\frac{x_{r_{s}}}{x_{0}}\Big)_{b+2}\right]=\left(\frac{x_{r_{s}}}{x_{r_{s}}q^{k_{s}}}\right)_{b+2}
=(q^{-k_{s}})_{b+2}=(q^{-b-1})_{b+2}=0.
\end{equation*}
If $r_s>m$, then $Q(h\mid \mathbf{r};\mathbf{k})$ has the factor
\begin{equation*}
E_{\mathbf{r},\mathbf{k}}\left[\Big(\frac{x_{r_{s}}q}{x_{0}}\Big)_{b+\chi(r_s\geqslant n-l+1)}\right]=
E_{\mathbf{r},\mathbf{k}}\left[\Big(\frac{x_{r_{s}}q}{x_{0}}\Big)_{b+1}\right]=\left(\frac{x_{r_{s}}q}{x_{r_{s}}q^{k_{s}}}\right)_{b+1}
=(q^{1-k_{s}})_{b+1}=(q^{-b})_{b+1}=0.
\end{equation*}

\emph{Proof of property (ii)}. Note that since $h \geqslant k_i$ for
all $i$ and $h\in D_1\bigcup D_2 \bigcup D_3 \bigcup \{(n-l-1)k+b+1\}$, the hypothesis implies that $h > sk-\chi(s<m)$.

We only show that $Q(h\mid \mathbf{r};\mathbf{k})$ is proper in
$x_{r_s}$ so that Lemma \ref{lem-almostprop} applies. The rest is the same as that in the proof of \cite[Lemma 5.1]{gess-xin2006}.
To this end we write $Q(h\mid \mathbf{r};\mathbf{k})$ as
$N/D$, in which $N$ (the numerator) is
\begin{equation*}
E_{\mathbf{r};\mathbf{k}}\bigg[\prod_{i=1}^{m}\Big(\frac{x_{i}}{x_{0}}\Big)_{b+1+\chi(i\geqslant n-l+1)}\prod_{i=m+1}^{n}\Big(\frac{x_{i}q}{x_{0}}\Big)_{b+\chi(i\geqslant n-l+1)}
\cdot \prod_{\substack{1\leqslant i,j\leqslant n\\ i\neq
j}}\Big(\frac{x_{i}}{x_{j}}q^{\chi(i>j)}\Big)_{k}\bigg],
\end{equation*}
and $D$ (the denominator) is
\begin{equation*}
E_{\mathbf{r};\mathbf{k}}\left[\prod_{i=1}^{m}\Big(\frac{x_{0}}{x_{i}q^{h}}\Big)_{h+1}
\prod_{i=m+1}^{n}\Big(\frac{x_{0}}{x_{i}q^{h}}\Big)_{h}\Big/\prod_{i=1}^{s}\Big(1-\frac{x_{0}}{x_{r_{i}}q^{k_{i}}}\Big)\right].
\end{equation*}
Now let $R = \{r_0, r_1, \ldots , r_s\}$. Then the degree in $x_{r_s}$
of
\begin{equation*}
E_{\mathbf{r};\mathbf{k}}\left[\Big(1-\frac{x_{i}}{x_{j}}q^{l}\Big)\right]
\end{equation*}
is 1 if $i\in R$ and $j\notin R$, and is 0 otherwise, as is easily
seen by checking the four cases. Thus the part of $N$ contributing
to the degree in $x_{r_s}$ is
\begin{equation*}
E_{\mathbf{r};\mathbf{k}}\left[\prod_{i=1}^{s}\prod_{j\neq
r_{0},\ldots,r_{s}}\Big(\frac{x_{r_{i}}}{x_{j}}q^{\chi(r_{i}>j)}\Big)_{k}\right],
\end{equation*}
which has degree $(n-s)sk$, and the part of $D$ contributing to the
degree in $x_{r_{s}}$ is
\begin{equation*}
E_{\mathbf{r};\mathbf{k}}\left[\prod_{\begin{subarray}{c}j\neq
r_{0},\ldots,r_{s}\\ j\leqslant m \end{subarray}}\Big(\frac{x_{0}}{x_{j}q^{h}}\Big)_{h+1}
\prod_{\begin{subarray}{c}j\neq
r_{0},\ldots,r_{s}\\ j>m \end{subarray}}\Big(\frac{x_{0}}{x_{j}q^{h}}\Big)_{h}\right],
\end{equation*}
which has degree at least $(n-s)h+\chi(s<m)$.

Thus the total degree of $Q(h\mid \mathbf{r};\mathbf{k})$ in
$x_{r_{s}}$ is at most
$$(n-s)(sk-h)-\chi(s<m)\le (n-s)(\chi(s<m)-1)-\chi(s<m) <0,$$
 so $Q(h\mid \mathbf{r};\mathbf{k})$
is proper in $x_{r_{s}}$.
\qed

Now we are ready to prove the vanishing lemma.
\begin{proof}[Proof of the vanishing lemma]
Recall that $\CT_{\mathbf{x}}Q(h)=M'_n(-a,b,k,m,l;q)$.
We prove by induction on $n-s$ that
$$\CT_{\mathbf{x}} Q(h \mid \mathbf{r};\mathbf{k}) = 0;$$
the lemma is the case $s=0$. (Note that taking the constant
term with respect to a variable that does not appear has no effect.)
We may assume that $s\leqslant n$ and $0<r_1<\cdots<r_s\leqslant n$, since
otherwise $Q(h \mid \mathbf{r};\mathbf{k})$ is not defined. If $s=n$ then $r_i$ must equal $i$
for $i=1,\dots ,n$ and thus $Q(h \mid \mathbf{r};\mathbf{k})=Q(h\mid 1,2,\dots, n;
k_{1}, k_{2},\dots, k_{n})$, which by property (i) of Lemma
\ref{lem-lem} is 0, since for each $i$, $k_i\leqslant h\leqslant (n-1)k+b+\chi(l>0)$.
Now suppose that $0\leqslant s<n$. Applying $\CT_{\mathbf{x}}$ to both
sides of \eqref{eq-s-s+1} gives
$$
\CT_{\mathbf{x}}Q(h \mid \mathbf{r};\mathbf{k})=\sum_{r_s < r_{s+1}\leqslant n\atop 0\leqslant k_{{s+1}}\leqslant h}
    \CT_{\mathbf{x}}Q(h\mid r_1,\dots, r_s, r_{s+1};k_1,\dots, k_s,k_{s+1})
$$
when property (ii) of Lemma \ref{lem-lem} applies. Thus by  Lemma
\ref{lem-lem}, $\CT_{\mathbf{x}}Q(h \mid \mathbf{r};\mathbf{k})$ is either~0 or is a sum of terms,
all of which are 0 by induction.
\end{proof}

\section{Proof for the extra point}\label{s-extra}

We need the following lemma.
\begin{lem}\label{lem-extra1}
Assume Theorem \ref{t-HabsigerKadell} holds for smaller values of $n$. Let $h=(n-l-1)k+b+1$. If $0\leqslant m,l <n$ and $k>b+1$ then
\begin{align*}
\CT_x Q(h\mid  1,\dots,n-l;h,h-k\dots,b+1)
=M_n(-h,b,k,m,l;q).
\end{align*}
\end{lem}
\begin{proof}
We have to split into the following two cases.

\noindent
Case 1: $m+l\leqslant n$. Then
\begin{align}\label{e-Qh0}
& Q(h\mid  1,\dots,n-l;h,h-k,\dots,b+1)=D\cdot (A\cdot B) \cdot \prod_{n-l+1\leqslant i<j\leqslant n} \left(\frac{x_{i}}{x_{j}}\right)_{k}
\left(\frac{x_{j}}{x_{i}}q\right)_{k},
\end{align}
where
\begin{align}
D=&\prod_{i=1}^m \frac{(q^{-(n-l-i)k-b-1})_{b+1}}{(q)_{(n-l-i)k+b+1}(q^{-(i-1)k})_{(i-1)k}}
\prod_{i=m+1}^{n-l} \frac{(q^{-(n-l-i)k-b})_{b}}{(q)_{(n-l-i)k+b}(q^{-(i-1)k})_{(i-1)k}} \nonumber \\
&\qquad \times \prod_{1\leqslant i<j \leqslant n-l}
(q^{(i-j)k})_k (q^{(j-i)k+1})_k \nonumber \\
=&\prod_{i=1}^m \frac{(q^{-(n-l-i)k-b-1})_{b+1}}{(q)_{(n-l-i)k+b+1}}
\prod_{i=m+1}^{n-l} \frac{(q^{-(n-l-i)k-b})_{b}}{(q)_{(n-l-i)k+b}}
\prod_{j=1}^{n-l}\frac{\prod_{i=1}^{j-1}(q^{(i-j)k})_k (q^{(j-i)k+1})_k }{(q^{-(j-1)k})_{(j-1)k}} \nonumber \\
=&\prod_{i=1}^m \frac{(q^{-(n-l-i)k-b-1})_{b+1}}{(q)_{(n-l-i)k+b+1}}
\prod_{i=m+1}^{n-l} \frac{(q^{-(n-l-i)k-b})_{b}}{(q)_{(n-l-i)k+b}}
\prod_{j=1}^{n-l}\frac{(q)_{jk}}{(q)_k},
\end{align}
\begin{align}
A=\prod_{i=n-l+1}^n \frac{\left(\frac{x_i}{x_{n-l}}q^{-b}\right)_{b+1}}
{\left(\frac{x_{n-l}}{x_i}q^{-(n-l-1)k}\right)_{(n-l-1)k+b+1}},
\end{align}
and
\begin{align}
B=&\prod_{\begin{subarray}{l} 1\leqslant i\leqslant n-l \\ n-l+1\leqslant
j\leqslant n\end{subarray}} \left(\frac{x_{n-l}}{x_j}q^{-(n-l-i)k}\right)_k
\left(\frac{x_{j}}{x_{n-l}}q^{(n-l-i)k+1}\right)_k \nonumber \\
=&\prod_{j=n-l+1}^n \left(\frac{x_{n-l}}{x_j}q^{-(n-l-1)k}\right)_{(n-l)k}
\left(\frac{x_{j}}{x_{n-l}}q\right)_{(n-l)k}.
\end{align}
For $k>b+1$, after cancelations and combinations, we obtain
\begin{align}
&\CT_{x}A\cdot B\cdot\prod_{n-l+1\leqslant i<j\leqslant n} \left(\frac{x_{i}}{x_{j}}\right)_{k}
\left(\frac{x_{j}}{x_{i}}q\right)_{k} \nonumber \\
=&\CT_x \prod_{i=n-l+1}^n
\left(\frac{x_{n-l}}{x_{i}}q^{b+1}\right)_{k-b-1}
\left(\frac{x_i}{x_{n-l}}q^{-b}\right)_{(n-l)k+b+1}
\prod_{n-l+1\leqslant i<j\leqslant n} \left(\frac{x_{i}}{x_{j}}\right)_{k}
\left(\frac{x_{j}}{x_{i}}q\right)_{k} \nonumber \\
=&\CT_x \prod_{i=n-l+1}^n
\left(\frac{x_{n-l}}{x_{i}}\right)_{k-b-1}
\left(\frac{x_i}{x_{n-l}}q\right)_{(n-l)k+b+1}
\prod_{n-l+1\leqslant i<j\leqslant n} \left(\frac{x_{i}}{x_{j}}\right)_{k}
\left(\frac{x_{j}}{x_{i}}q\right)_{k},
\end{align}
where the last equality is obtained by making the substitution $x_{n-l}=x_{n-l}q^{-b-1}$.
This is just $M'_{l}(k-b-1,(n-l)k+b+1,k,0,0;q)$.
By Remark \ref{remark1} (or the hypothesis), we obtain
 \begin{align}
\CT_{x}Q(h)
=&D\cdot \prod_{i=0}^{l-1} \frac{(q)_{(n-l+i+1)k}(q)_{(i+1)k}}{(q)_{(i+1)k-b-1}(q)_{(n-l+i)k+b+1}(q)_{k}},
\end{align}
which can be routinely checked to be equal to
$M_n(-h,b,k,m,l;q).$

\medskip
\noindent
Case 2: $m+l> n$. The computation is similar to but more complicated than case 1. Indeed we need the case 1 result in some sense.
We omit some details for brevity.
We have
\begin{align}\label{e-Qh1}
Q(h\mid  1,\dots,n-l;h,h-k,\dots,b+1)=D'\cdot  (A'\cdot B') \cdot \prod_{n-l+1\leqslant i<j\leqslant n} \left(\frac{x_{i}}{x_{j}}\right)_{k}
\left(\frac{x_{j}}{x_{i}}q\right)_{k},
\end{align}
where
\begin{align}
D'=&\prod_{i=1}^{n-l} \frac{(q^{-(n-l-i)k-b-1})_{b+1}}{(q)_{(n-l-i)k+b+1}(q^{-(i-1)k})_{(i-1)k}}
\cdot \prod_{1\leqslant i<j \leqslant n-l}
(q^{(i-j)k})_k (q^{(j-i)k+1})_k  \nonumber \\
=&\prod_{i=1}^{n-l} \frac{(q^{-(n-l-i)k-b-1})_{b+1}(q)_{ik}}{(q)_{(n-l-i)k+b+1}(q)_{k}},
\end{align}
and $A'$ and $B'$ are similar to $A$ and $B$, with $A'B'$ simplifies as
\begin{align}
A'\cdot B' = &\prod_{i=n-l+1}^m
\left(\frac{x_{n-l}}{x_{i}}q^{b+2}\right)_{k-b-2}
\left(\frac{x_i}{x_{n-l}}q^{-b-1}\right)_{(n-l)k+b+2}
\nonumber \\
&\times \prod_{i=m+1}^n
\left(\frac{x_{n-l}}{x_{i}}q^{b+1}\right)_{k-b-1}
 \left(\frac{x_i}{x_{n-l}}q^{-b}\right)_{(n-l)k+b+1}.
\end{align}

A similar computation gives
\begin{align*}
&\CT_{x}A'\cdot B'\cdot \prod_{n-l+1\leqslant i<j\leqslant n} \left(\frac{x_{i}}{x_{j}}\right)_{k}
\left(\frac{x_{j}}{x_{i}}q\right)_{k} =M'_{l}(k-b-1,(n-l)k+b+1,k,m-n+l,0;q),
\end{align*}
which is the constant term in \eqref{eq-qMorris} in case 1, and is known to be $ M_{l}(k-b-1,(n-l)k+b+1,k,m-n+l,0;q)$ by the hypothesis.
Then it only left to show that
$$D'\cdot \prod_{i=0}^{l-1} \frac{(q)_{(n-l+i+1)k}(q)_{(i+1)k}}
{(q)_{(i+1)k-b-1-\chi(i<m-n+l)}(q)_{(n-l+i)k+b+1+\chi(i\geqslant n-m)}(q)_{k}}
=M_n(-h,b,k,m,l;q),$$
which is routine.
\end{proof}
Note that we can avoid using the induction hypothesis. The truth of Lemma \ref{lem-extra1} in case 1 results in the truth of Theorem \ref{t-HabsigerKadell} in case 1, which is needed in the case 2 of Lemma \ref{lem-extra1}.

Now we are ready to deal with the extra point.
\begin{proof}[Proof of Lemma \ref{lem-extra}]
As we discussed in \eqref{eq-Qx0},
$\CT_{x_0}Q(h)$ can be written as
\begin{equation}
\CT_{x_{0}}Q(h)=\sum_{\begin{subarray}{l} 1\leqslant r_{1}\leqslant n \\ 1\leqslant
k_{1}\leqslant h\end{subarray}}Q(h\mid r_{1},k_{1}).
\end{equation}
Iteratively apply Lemma \ref{lem-lem} to each summand if
applicable. Finally we get
\begin{align*}
\CT_{\mathbf{x}}Q(h)=\CT_{\mathbf{x}}\sum_{r_1,\dots,r_s,
k_1,\dots,k_s}Q(h\mid r_1,\dots,r_s;k_1,\dots,k_s),
\end{align*}
where the sum ranges over all the $r$'s and the $k$'s with $0<r_1<\cdots
<r_s\leqslant n, 0\leqslant k_1,k_2,\dots,k_s\leqslant h$ such that Lemma
\ref{lem-lem} does not apply. Note that we may have different $s$.

Since $h=(n-l-1)k+b+1$ and $0\leqslant k_i\leqslant h$, by Lemma \ref{lem-important},
there leaves only one term for which Lemma \ref{lem-lem} is not applicable. This term corresponds to
$k_i=(n-l-i)k+b+1$ for $i=1,\ldots,n-l$ and $r_i=i$ for $i=1,\ldots,n-l$.
It follows that
\begin{align}
\CT_x Q(h)=&\CT_x Q(h\mid  1,\dots,n-l;h,h-k,\dots,b+1).
\end{align}
The lemma then follows from Lemma \ref{lem-extra1}.
\end{proof}

The extra point $h=(n-l-1)k+b+1$ in Lemma \ref{lem-extra} is not easy to find. This $h$
seems to be the only choice of the extra point for which it is not hard to show that
$\CT_xQ(h)=M_n(-h,b,k,m,l;q).$ Intuitively a desired extra point must be chosen from boundary values,
i.e., values next to the vanishing points listed in \eqref{root}.

Firstly, the boundary values $h=(n-l-2)k+b+1,(n-l-3)k+b+1,\ldots,b+1$ do not work.
To see this, take $n=3,m=l=1$ for example. Then we can only get
{\small \begin{align*}
Q(b+1)=&(-1)^{b+1}q^{-\binom{b+2}{2}}\CT_x\frac{\left(1/x_2\right)_{k-b-1}\left(x_2q\right)_{b+k+1}
\left(1/x_3\right)_{k-b-1}\left(x_3q\right)_{b+k+1}}{1-x_2q^{b+1}}
\cdot \left(\frac{x_2}{x_3}\right)_k\left(\frac{x_3}{x_2}q\right)_k \nonumber \\
+&(-1)^bq^{-\binom{b+1}{2}}\CT_x\frac{\left(q/x_1\right)_{k-b-1}\left(x_1\right)_{b+k+1}
\left(1/x_3\right)_{k-b-1}\left(x_3q\right)_{b+k+1}}
{1-1/(x_1q^{b+1})}
\left(\frac{x_1}{x_3}\right)_k\left(\frac{x_3}{x_1}q\right)_k.
\end{align*}}

Secondly, the boundary values $h=mk,(m+1)k,\ldots,(n-1)k$ do not work either for a similar reason.

\subsection*{Acknowledgements}
The authors would like to thank Christine Krattenthaler for pointing out a trivial identity in the earlier draft. This work was supported by the
National Science Foundation of China.

\end{document}